\documentclass[sn-mathphys,Numbered]{sn-jnl}

\usepackage{graphicx}%
\usepackage{multirow}%
\usepackage{amsmath,amssymb,amsfonts}%
\usepackage{amsthm}%
\usepackage{mathrsfs}%
\usepackage[title]{appendix}%
\usepackage{xcolor}%
\usepackage{textcomp}%
\usepackage{manyfoot}%
\usepackage{booktabs}%
\usepackage{algorithm}%
\usepackage{algorithmicx}%
\usepackage{algpseudocode}%
\usepackage{listings}%
\usepackage{placeins}

\theoremstyle{thmstyleone}%
\newtheorem{theorem}{Theorem}
\theoremstyle{thmstyletwo}%
\newtheorem{remark}{Remark}%
\newtheorem{lemma}{Lemma}%
\theoremstyle{thmstylethree}%

\begin{document}

\title[Article Title]{Linear Filtering for Discrete Time Systems Driven by Fractional Noises}

\author[1]{\fnm{Yuecai} \sur{Han}}\email{hanyc@jlu.edu.cn}

\author*[1]{\fnm{Yuhang} \sur{Li}}\email{yuhangl22@mails.jlu.edu.cn}
\equalcont{These authors contributed equally to this work.}

\affil[1]{\orgdiv{School of Mathematics}, \orgname{Jilin University}, \orgaddress{ \city{Changchun}, \postcode{130012},  \state{Jilin Province}, \country{China}}}


\abstract{In this paper, we study the discrete time filtering problems for linear systems driven by fractional noises. The main difficulty comes from the non-Markovian of the noises. We construct the difference equation of the covariance process through the properties of the noises and transform the filtering problem to an optimal control problem.  We obtain the necessary condition that the coefficients of the optimal filter should satisfy and show there is no consistent optimal filter, which is a significant difference from the classical Kalman filter. Finally, a simple example is considered to illustrate the main results. Further more, our method could also deal with systems driven by any other colored noises, as long as the self-covariation function is known. }




\maketitle

\section{Introduction}

Based on \cite{kalman1961neww} and related works, classical Kalman filter has been popularly investigated and applied in many areas due to the naturally arise of observation outliers, for example, see \cite{chui2017kalman,lawrence1971kalman,masreliez1975approximate,tang2020robot}. For research on theoretical aspects, stochastic non-linear filter is first studied by Kushner \cite{kushner1964dynamical,kushner1967dynamical} and Stratonovich \cite{stratonovich1960conditional}. Fujisaki et al. \cite{Fujisaki1972stochastic} point out the optimal filter satisfy a non-linear stochastic partial differential equation (SPDE), which is called   Kushner–FKK equation. Kallianpur and Striebel \cite{kallianpur1968estimation,kallianpur1969stochastic} establish the representation of the``unnormalized filter". The linear SPDE governing the dynamics of the``unnormalized filter" is investigated in \cite{duncan1967probability,mortensen1966optimal,zakai1969optimal}, which is called Duncan-Mortensen-Zakai
equation or Zakai's equation.

To model
phenomena that exhibit long range dependence, fractional Brownian motion is popularly investigated, which is originally introduced by Mandelbrot and Van Ness \cite{mandelbrot1968fractional}. Let $H\in(1/2,\,1)$ be a fixed constant. The $m$-dimensional fractional Brownian motion $B_t^H=\left(B_1^H(t),\cdots,B_m^H(t)\right), t\in[0,T]$  of Hurst parameter $H$ is a continuous, mean 0 Gaussian process with the covariance
\begin{align}\label{covfbm}
\mathbb{E}[B_i^H(t)B_j^H(s)]=\frac{1}{2}\delta_{ij}(t^{2H}+s^{2H}-| t-s|^{2H}),
\end{align}
where
\begin{equation}
\delta_{ij}=\left\{
\begin{aligned}
&1,\ \ {\rm if}\ i=j,
\\
&0,\ \ {\rm if}\ i\neq j, 
\end{aligned}
\right.
\notag
\end{equation}
$i,\,j=1,\dots,m$. Moreover, it could be generated by a standard Brownian motion through
\begin{align*}
B_j^H(t)=\int_0^t Z_H(t,\,s) dB_j(s), \quad 1\le j\le m,
\end{align*}
where 
\begin{equation*}
Z_H(t,\,s)=\kappa_H\left[\left(\frac{t}{s}\right)^{H-\frac{1}{2}}(t-s)^{H-\frac{1}{2}}-\left(H-\frac{1}{2}\right)s^{\frac{1}{2}-H}\int_s^t u^{H-\frac{3}{2}}(u-s)^{H-\frac{1}{2}}du\right]
\end{equation*}
with
\begin{equation*}
\kappa_H=\sqrt{\frac{2H\varGamma(\frac{3}{2}-H)}{\varGamma(H+\frac{1}{2})\varGamma(2-2H)}}.
\end{equation*}

Filtering problems for systems driven fractional noises are studied and applied in many areas. Linn and Amirdjanova\cite{linn2009representations} investigate the problem of nonlinear filtering of multiparameter random fields, which observed in the presence of a long-range dependent spatial noise. Su et al.\cite{su2011multiscale} presents a new ultrawideband multi-scale Kalman filter
algorithm to suppress the interference of the strong fractional noise signal. 
Afterman et al.\cite{afterman2022linear}  develop a method of asymptotic analysis of the integro-differential filtering equations driven by fractional Brownian motions. But as we know, there is not any work obtains the explicit condition that the optimal filtering should satisfy for discrete time systems driven by fractional noises.

In this paper, we study the optimal filtering problem for coupled systems driven by fractional noises:
\begin{equation*}
    \left\{\begin{array}{l}
x(k+1)=A(k)x(k)+C(k)y(k)+\sigma(k)W_1(k),\\
y(k+1)=D(k)x(k)+F(k)y(k)+\gamma(k)W_2(k),\\
x(0)=x_0,\\
y(0)=0.
\end{array}\right.
\end{equation*}
Here process $y$ is observable, which contains partial information of $x$. Then we try to construct a processes $z$ to estimate the unobservable signals by using $y$. 

The main difficulty to this problem comes from the non-Markovian of the noises, which makes it is harder to estimate the covariance of the error process. Through the properties of the noises, we use variation methods (used in \cite{ahmed1991quadratic,ahmed2002filtering}) to get the necessary condition the optimal filter should satisfy. Compared with filtering problem for continuous systems driven by fractional noises investigated in \cite{ahmed2002filtering}, we study it in discrete time view and 
concretely show the significant difference with classical Kalman filter: there is no consistent optimal filter and the coefficients of the optimal filter are influenced by the Hurst parameters of the noises.

The rest of this paper is organized as follows. In section 2, we introduce the state and measurement dynamics, then we state the form of the filter we considered. In section 3, we deal with the error process and show the dynamic of the covariance of it, so that the filtering problem transform to a control problem. In section 4, we give the necessary condition that the optimal control should satisfy, and the filtering problem is solved at the same time. In section 4, an example is considered to illustrate main results.

\section{Dynamics and filtering problem}
Consider the system governed as follows
\begin{equation}\label{x}
    \left\{\begin{array}{l}
x(k+1)=A(k)x(k)+C(k)y(k)+\sigma(k)W_1(k),\\
y(k+1)=D(k)x(k)+F(k)y(k)+\gamma(k)W_2(k),\\
x(0)=x,\\
y(0)=0.
\end{array}\right.
\end{equation}
Here we assume the process $\{x(\cdot),y(\cdot)\}$ taking values in $\mathbf{R}^n$ and $\mathbf{R}^m$, respectively. $x_0$ is a random vector independent with noises, 
$W_1(k)=B^{H_1}(k+1)-B^{H_1}(k)$ and $W_2(k)=\tilde{B}^{H_2}(k+1)-\tilde{B}^{H_2}(k)$, where $B^{H_1}$ and $\tilde{B}^{H_2}$ are independent fractional Brownian motions taking value in $\mathbf{R}^d$ with Hurst parameter $H_1, H_2\in(1/2,1)$, respectively. Through (\ref{covfbm}), it is not difficult to show $\sum_{i=k}^{+\infty}\mathbb{E}\|W_1(k)W_1(i)\|=\sum_{i=k}^{+\infty}\mathbb{E}\|W_2(k)W_2(i)\|=+\infty, \forall k\ge 0$, which describe the long range phenomena. 
It is clear that the matrices $\{A,C,\sigma,D,F,\gamma\}$ take values in $\mathbf{R}^{n\times n},\mathbf{R}^{n\times n},\mathbf{R}^{n\times d},\mathbf{R}^{m\times n},\mathbf{R}^{m\times n},\mathbf{R}^{m\times d},$ respectively. 

Denote $\mathcal{F}^y_k:=\sigma(y(0),y(1),...,y(k))$. The aim is to find a process $z$ such that $z(k)$ is $\mathcal{F}^y_k$-adapted process satisfying
\begin{align*}
\mathbb{E}z(k)=\mathbb{E}x(k),
\end{align*}
and hope to minimize $\mathbb{E}\|x(k)-z(k)\|^2$. 

\begin{remark}
It is well known that $z$ should be given by $z(k)=\mathbb{E}[x(k)|\mathcal{F}_k^y]$, but it is quite difficult. So our objective here is to find the best unbiased minimum variance (UMV) linear filter driven by the observed process $y$, as described by the following dynamic:
\begin{align}\label{z}
    \left\{\begin{array}{l}
z(k+1)=H(k)z(k)+M(k)y(k)+\Gamma(k)y(k+1),\\
z(0)=\mathbb{E}x_0,
\end{array}\right.
\end{align}
which minimize
\begin{align*}
\sum_{k=0}^Na(k)\mathbb{E}\|x(k)-z(k)\|^2.
\end{align*}
Here $H,M,\Gamma$ are suitable matrix-valued functions to be determined and $\{a(\cdot)\}$ is a given non-negative sequence.
\end{remark}

\section{Reformulation of the filtering problem as a optimal control problem}

To simplify the notation without losing the generality, we consider the case $n=m=d=1$. Define
\begin{align*}
e(k)=x(k)-z(k),\quad k=0,1,...,N
\end{align*}
to describe the variance of the filter. Here $x$ is the solution of equation (\ref{x}) and $z$ is the solution of equation (\ref{z}) corresponding to any choice of $H,M,\Gamma$. Then the error process $e$ satisfy the following stochastic difference equation:
\begin{align}\label{de}
    \left\{\begin{array}{l}
e(k+1)=[A(k)-\Gamma(k)D(k)]e(k)+[C(k)-\Gamma(k)F(k)-M(k)]y(k)\\\\\qquad\qquad+[A(k)-\Gamma(k)D(k)-H(k)]z(k)+\sigma(k)W_1(k)-\Gamma(k)\gamma(k)W_2(k),\\
\\e(0)=x_0-\mathbb{E}x_0.
\end{array}\right.
\end{align}
Since the estimate is expected to be unbiased, we determine $H,M$ by $\Gamma$ and denote
\begin{align}\label{hm}
H_\Gamma(k)=A(k)-\Gamma(k)D(k),\notag\\
M_\Gamma(k)=C(k)-\Gamma(k)F(k).
\end{align}
Then the solution to $e$ is given by
\begin{align*}
e(k)=H_\Gamma(k-1,0)e(0)+\sum_{i=0}^{k-1}\sigma(i)H_\Gamma(k-1,i+1)W_1(i)+\sum_{i=0}^{k-1}\Gamma(i)\gamma(i)H_\Gamma(k-1,i+1)W_2(i),
\end{align*}
where $H_\Gamma(k-1,i)=\Pi_{j=i}^{k-1}H_\Gamma(j)$.

~\\

Define the error covariance $K(k)=\mathbb{E}[e^2(k)]$. Trough the properties of the noises, we obtain that
\begin{align}\label{kk}
K(k)=&H_\Gamma^2(k-1,0)\mathbb{V}(x_0)\notag\\
&
+\sum_{i=0}^{k-1}\sum_{j=0}^{k-1}\left[\rho_1(i-j)\sigma(i)\sigma(j)+\rho_2(i-j)\Gamma(i)\Gamma(j)\gamma(i)\gamma(j)\right]\notag\\
&\qquad\qquad\qquad\times H_\Gamma(k-1,i+1)H_\Gamma(k-1,j+1),
\end{align}
where $\rho_l(k)=\mathbb{E}[W_l(i)W_l(i+k)]=|k+1|^{2H_l}+|k-1|^{2H_l}-2|k|^{2H_l}, l=1,2$.

Now we formulate the ﬁltering problem as a control problem. First we recall that equation (\ref{z}), with $\Gamma$ to be determined and $H, M$ are determined by $\Gamma$ through equation (\ref{hm}), gives an unbiased estimate of $x$ and the covariance of the estimate is given by equation (\ref{kk}). Then we choose $\Gamma$ to get a filter with minimum variance estimate, which means $\sum_{k=0}^NK(k)$ is minimum. In this paper, we consider a more general case. 
Our aim is to minimize $\sum_{k=0}^Na(k)K(k)$
, where $a$ is any given real positive deﬁnite symmetric
matrix-valued function. Then the optimal filtering
problem is equivalent to the optimal control problem: ﬁnd $\Gamma$ that imparts a minimum to the
functional $J$ subject to the dynamic constraint equation (\ref{kk}).

\section{The optimal filter}
In this section, we focus on solving the following control problem: the state equation is
\begin{align}\label{state}
K(k)=&H_\Gamma^2(k-1,0)\mathbb{V}(x_0)+\sum_{i=0}^{k-1}H_\Gamma^2(k-1,i+1)\left[\sigma^2(i)+\Gamma^2(i)\gamma^2(i)\right]\notag\\
&+2\sum_{i=0}^{k-1}\sum_{j=0}^{i-1}\rho_1(i-j)\sigma(i)\sigma(j)H_\Gamma(k-1,i+1)H_\Gamma(k-1,j+1)\notag\\&
+2\sum_{i=0}^{k-1}\sum_{j=0}^{i-1}\rho_2(i-j)\Gamma(i)\Gamma(j)\gamma(i)\gamma(j)H_\Gamma(k-1,i+1)H_\Gamma(k-1,j+1).
\end{align}
The cost function is  defined as
\begin{align}\label{cost}
J(\Gamma)=\sum_{k=0}^Na(k)K(k).
\end{align}

To obtain the optimal control and the optimal ﬁlter, we use the variation technique. So the $G\hat{a}teaux$ derivative of $K$ with respect to $\Gamma$  is needed. Thus, we state the following results.
~\\

\begin{lemma}
Let $\tilde{\Theta}$ denote the $G\hat{a}teaux$ derivative of the map $\Gamma\to \Theta_\Gamma$ at $\Gamma_0$ in the direction $\beta$, where $\beta$ denotes $\Gamma-\Gamma_0$ for any $\Gamma\in L^\infty\big([0,T],\mathbf{R}\big)$. Then we have
\begin{align*}
\tilde{H}(k,i)=-\sum_{j=i}^{k-1}H_{\Gamma_0}(j-1,i)H_{\Gamma_0}(k-1,j+1)D(j)\beta(j).
\end{align*}
\end{lemma}

\begin{proof}
Notice that $H_\Gamma(k,i)=H_\Gamma(k)H_\Gamma(k-1,i),\, 0\le i<k\le N-1$ and $H_\Gamma(i-1,i)=1,\,0\le i\le N-1$. Thus, we have
\begin{align*}
    \left\{\begin{array}{l}
\tilde{H}(k,i)=H_{\Gamma_0}(k-1)\tilde{H}(k-1,i)-H_{\Gamma_0}(k-1,i)D(k)\beta(k),\\
\\\tilde{H}(i-1,i)=0.
\end{array}\right.
\end{align*}
Then, by induction, we derive that 
\begin{align*}
\tilde{H}(k,i)=-\sum_{j=i}^{k-1}H_{\Gamma_0}(j-1,i)H_{\Gamma_0}(k-1,j+1)D(j)\beta(j).
\end{align*}
\end{proof}
Now we give the $G\hat{a}teaux$ derivative of $K$. 

\begin{theorem}\label{tildeKK}
The $G\hat{a}teaux$ derivative of $K$ at $\Gamma_0$ in the direction $\beta$ is
\begin{align}\label{tildek}
\tilde{K}(k)=\sum_{i=0}^{k-1}Q(k-1,i)\beta(i), \quad 1\le k\le N,
\end{align}
where
\begin{align*}
Q(k-1,i)=&-2H_{\Gamma_0}(k-1,0)\mathbb{V}(x_0)H_{\Gamma_0}(i-1,0)H_{\Gamma_0}(k-1,i+1)\\
&+\sum_{j=0}^{k-1}\rho_2(i-j)\Gamma_0(j)\gamma(i)\gamma(j)H_{\Gamma_0}(k-1,i+1)H_{\Gamma_0}(k-1,j+1)\\
&-\sum_{l=0}^{i-1}\sum_{j=0}^{k-1}\left[\rho_1(l-j)\sigma(l)\sigma(j)+\rho_2(l-j)\Gamma_0(l)\Gamma_0(j)\gamma(l)\gamma(j)\right] \\&\qquad\qquad\times H_{\Gamma_0}(k-1,j+1)H_{\Gamma_0}(i-1,l+1)H_{\Gamma_0}(k-1,i+1)D(i).
\end{align*}
\end{theorem}

\begin{proof}
Rewrite equation (\ref{kk}) as
\begin{align*}
K(k)=I_1(k-1)+\sum_{i=0}^{k-1}\sum_{j=0}^{k-1}I_2(k-1,i,j),
\end{align*}
then we deal with them separately. 

For $I_1$,
\begin{align}\label{i11}
\tilde{I}_1(k-1)=&2H_{\Gamma_0}(k-1,0)\mathbb{V}(x_0)\tilde{H}(k-1,0)\notag\\
=&-\sum_{i=1}^{k-1}2H_{\Gamma_0}(k-1,0)\mathbb{V}(x_0)H_{\Gamma_0}(i-1,0)H_{\Gamma_0}(k-1,i+1)\beta(i)\notag\\
:=&\sum_{i=0}^{k-1}Q_1(k-1,i)\beta(i).
\end{align}

The second part is more complex, taking the $G\hat{a}teaux$ derivative, we have
\begin{align*}
\tilde{I}_2(k-1,i,j)=&I_{21}(k-1,i,j)\beta(i)+I_{21}(k-1,j,i)\beta(j)\\
&+I_{22}(k-1,i,j)+I_{22}(k-1,j,i),
\end{align*}
where
\begin{align*}
I_{21}(k-1,i,j)=\rho_2(i-j)\Gamma_0(j)\gamma(i)\gamma(j)H_{\Gamma_0}(k-1,i+1)H_{\Gamma_0}(k-1,j+1)
\end{align*}
and
\begin{align*}
I_{22}(k-1,i,j)=&\left[\rho_1(i-j)\sigma(i)\sigma(j)+\rho_2(i-j)\Gamma_0(i)\Gamma_0(j)\gamma(i)\gamma(j)\right]\\
&\times H_{\Gamma_0}(k-1,j+1)\tilde{H}(k-1,i+1)\\
=&-\left[\rho_1(i-j)\sigma(i)\sigma(j)+\rho_2(i-j)\Gamma_0(i)\Gamma_0(j)\gamma(i)\gamma(j)\right]\\
&\times H_{\Gamma_0}(k-1,j+1)\sum_{l=i+1}^{k-1}H_{\Gamma_0}(l-1,i+1)H_{\Gamma_0}(k-1,l+1)D(l)\beta(l).
\end{align*}
Then take summation separately, we have that
\begin{align}\label{i21}
&\sum_{i=0}^{k-1}\sum_{j=0}^{k-1}I_{21}(k-1,i,j)\beta(i)\notag\\
&=\sum_{i=0}^{k-1}\sum_{j=0}^{k-1}\rho_2(i-j)\Gamma_0(j)\gamma(i)\gamma(j)H_{\Gamma_0}(k-1,i+1)H_{\Gamma_0}(k-1,j+1)\beta(i)
\notag\\&:=\sum_{i=0}^{k-1}Q_2(k-1,i)\beta(i)
\end{align}
and
\begin{align}\label{i22}
&\sum_{i=0}^{k-1}\sum_{j=0}^{k-1}I_{22}(k-1,i,j)\notag\\&=-\sum_{i=0}^{k-1}\sum_{l=0}^{i-1}\sum_{j=0}^{k-1}\left[\rho_1(l-j)\sigma(l)\sigma(j)+\rho_2(l-j)\Gamma_0(l)\Gamma_0(j)\gamma(l)\gamma(j)\right] \notag\\&\qquad\qquad\qquad\times H_{\Gamma_0}(k-1,j+1)H_{\Gamma_0}(i-1,l+1)H_{\Gamma_0}(k-1,i+1)D(i)\beta(i)\notag\\
&:=\sum_{i=0}^{k-1}Q_3(k-1,i)\beta(i).
\end{align}
Through (\ref{i11}), (\ref{i21}), (\ref{i22}) and define $Q=Q_1+Q_2+Q_3$, we complete the proof.

\end{proof}

~\\

By Theorem \ref{tildeKK}, the  $G\hat{a}teaux$ derivative of $J$ is given by
\begin{align*}
\tilde{J}:=\frac{dJ\Big(\Gamma_0(\cdot)+\varepsilon \beta(\cdot)\Big)}{d\varepsilon}\Bigg|_{\varepsilon=0}=\sum_{k=1}^Na(k)\tilde{K}(k)=\sum_{i=0}^{N-1}\left(\sum_{k=i+1}^Na(k)Q(k-1,i)\right)\beta(i).
\end{align*}

Assume $\Gamma_0$ is the optimal control, then we have
$
\tilde{J}\ge 0
$ 
for any $\beta$. By the arbitrary of $\beta$, $\sum_{k=i+1}^Na(k)Q(k-1,i)=0$ for all $0\le i\le N-1$, then we conclude the following theorem:
\begin{theorem}
Assume that $\Gamma_0$ is the optimal control for the control system (\ref{state}), (\ref{cost}), then the following system of equations hold:
\begin{align}\label{nec}
    \left\{\begin{array}{l}
H_{\Gamma_0}(k)=A(k)-\Gamma_0(k)D(k),\\\\
H_{\Gamma_0}(k-1,i)=\Pi_{j=i}^{k-1}H_{\Gamma_0}(j),\\\\
Q_1(k-1,i)=-2H_{\Gamma_0}(k-1,0)\mathbb{V}(x_0)H_{\Gamma_0}(i-1,0)H_{\Gamma_0}(k-1,i+1),
\\\\
Q_2(k-1,i)=\sum_{j=0}^{k-1}\rho_2(i-j)\Gamma_0(j)\gamma(i)\gamma(j)H_{\Gamma_0}(k-1,i+1)H_{\Gamma_0}(k-1,j+1),\\\\
Q_3(k-1,i)=-\sum_{l=0}^{i-1}\sum_{j=0}^{k-1}\left[\rho_1(l-j)\sigma(l)\sigma(j)+\rho_2(l-j)\Gamma_0(l)\Gamma_0(j)\gamma(l)\gamma(j)\right] \\\\\qquad\qquad\qquad\times H_{\Gamma_0}(k-1,j+1)H_{\Gamma_0}(i-1,l+1)H_{\Gamma_0}(k-1,i+1)D(i),\\\\
Q=Q_1+Q_2+Q_3,
\\\\
\sum_{k=i+1}^Na(k)Q(k-1,i)=0,\quad i=0,1,...,N-1.
\end{array}\right.
\end{align}
At the same time, the optimal filter is given by
\begin{align*}
    \left\{\begin{array}{l}
z_0(k+1)=\left(A(k)-\Gamma_0(k)D(k)\right)z_0(k)+\left(C(k)-\Gamma_0(k)F(k)\right)y(k)+\Gamma_0(k)y(k+1),\\
z_0(0)=\mathbb{E}x_0.
\end{array}\right.
\end{align*}
\end{theorem}

\begin{remark}
Compared with classical filtering problem, the optimal filter we obtained could not make $\mathbb{E}\|x(k)-z(k)\|^2$  minimize for all $0\le k\le N-1$, because the noises are not Markovian. We could only get a necessary condition for the optimal filter should satisfy under the cost function (\ref{cost}). Moreover, the optimal filter may be different if we change the weight function $a(k)$ or the Hurst parameters $H_1,H_2$. To show these, we consider the following example. 
\end{remark}

\section{Example}
Consider the system governed as follows
\begin{equation*}
    \left\{\begin{array}{l}
x(k+1)=W_1(k),\\
y(k+1)=x(k)-W_2(k),\quad k=0,1,\\
x(0)=0,\\
y(0)=0.
\end{array}\right.
\end{equation*}
Here $W_1, W_2$ are fractional noises with Hurst parameter $H_0$ such that $\rho_1(1)=\rho_2(1)=\rho\in(0,1)$.

In this case, $\mathbb{V}(x_0)=0$. Moreover, through (\ref{nec}), we have 
$$H_{\Gamma_0}(0,0)=\Gamma_0(0),\quad H_{\Gamma_0}(1,0)=\Gamma_0(0)\Gamma_0(1),\quad H_{\Gamma_0}(1,1)=\Gamma_0(1),$$
and 
\begin{align*}
Q_1=Q_3(0,0)=&Q_3(1,0)=0,  \quad Q_2(0,0)=\Gamma_0(0), \quad Q_2(1,0)=\left(\rho+\Gamma_0(0)\right)\Gamma_0(1)^2, \\\\&Q_2(1,1)=\left(\rho\Gamma_0(0)+1\right)\Gamma_0(1), \quad Q_3(1,1)=1+\Gamma_0(0)^2.
\end{align*}
Then through Theorem \ref{tildeKK}, the $G\hat{a}teaux$ derivative of $K$ is given by $\tilde{K}(1)=\Gamma_0(0)\beta(0), $ and 
\begin{align*}
\tilde{K}(2)=\left(\rho+\Gamma_0(0)\right)\Gamma_0(1)^2\beta(0)+\left[\left(\rho\Gamma_0(0)+1\right)\Gamma_0(1)+\left(1+\Gamma_0(0)\right)^2\right]\beta(1).
\end{align*}
To make $K(1)$ minimize, $\Gamma_0(0)$ must be $0$. Then it is clear that under the condition $\Gamma_0(0)=0,$ there is not a $\Gamma_0(1)$ such that $\left(\rho+\Gamma_0(0)\right)\Gamma_0(1)^2=\left(\rho\Gamma_0(0)+1\right)\Gamma_0(1)+\left(1+\Gamma_0(0)\right)^2=0$, since $\rho\in(0,1)$.

Then we try to give the condition that $\Gamma_0$ should satisfy, through the last condition of (\ref{nec}), we have 
\begin{align*}
a(2)\left[\left(\rho\Gamma_0(0)+1\right)\Gamma_0(1)+\left(1+\Gamma_0(0)\right)^2\right]=0,
\end{align*}
and
\begin{align*}
a(1)\Gamma_0(0)+a(2)\left(\rho+\Gamma_0(0)\right)\Gamma_0(1)^2=0.
\end{align*}
So $\Gamma_0(1)=-\frac{(1+\Gamma_0(0))^2}{\rho\Gamma_0(0)+1}$ and $\Gamma_0(0)$ is the solution to the following equation:
\begin{align*}
a(1)\Gamma_0(0)\left(\rho\Gamma_0(0)+1\right)^2+a(2)(\rho+\Gamma_0(0))(1+\Gamma_0(0))^4=0,
\end{align*}
which shows the optimal filter may be different if we change $a(\cdot)$ or $\rho$.

\begin{figure}
    \centering
    \includegraphics[width=1\linewidth]{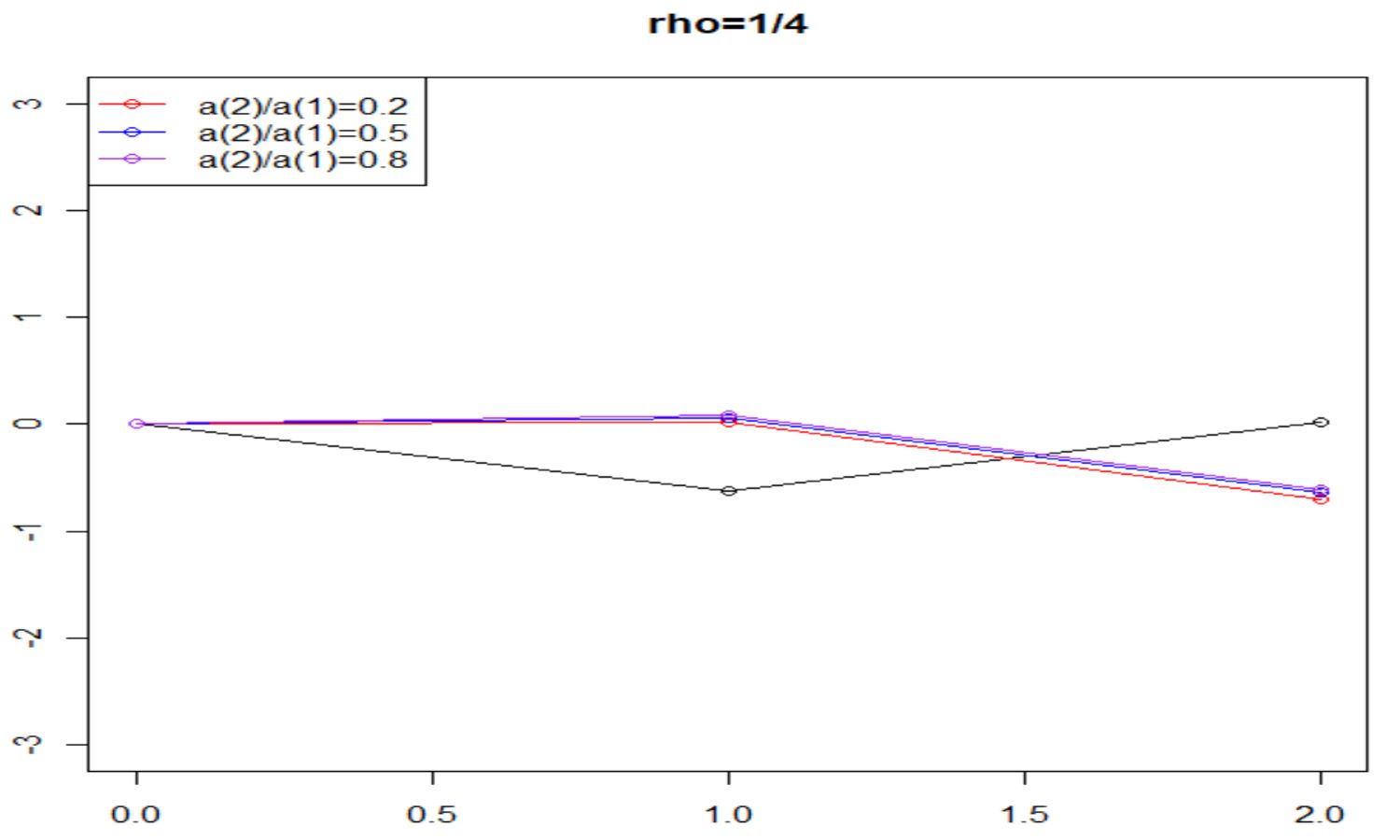}
    \caption{Filtering corresponding to different weight functions in the case $\rho=1/4$}
    \label{fig:enter-label}
\end{figure}

\begin{figure}
    \centering
    \includegraphics[width=1\linewidth]{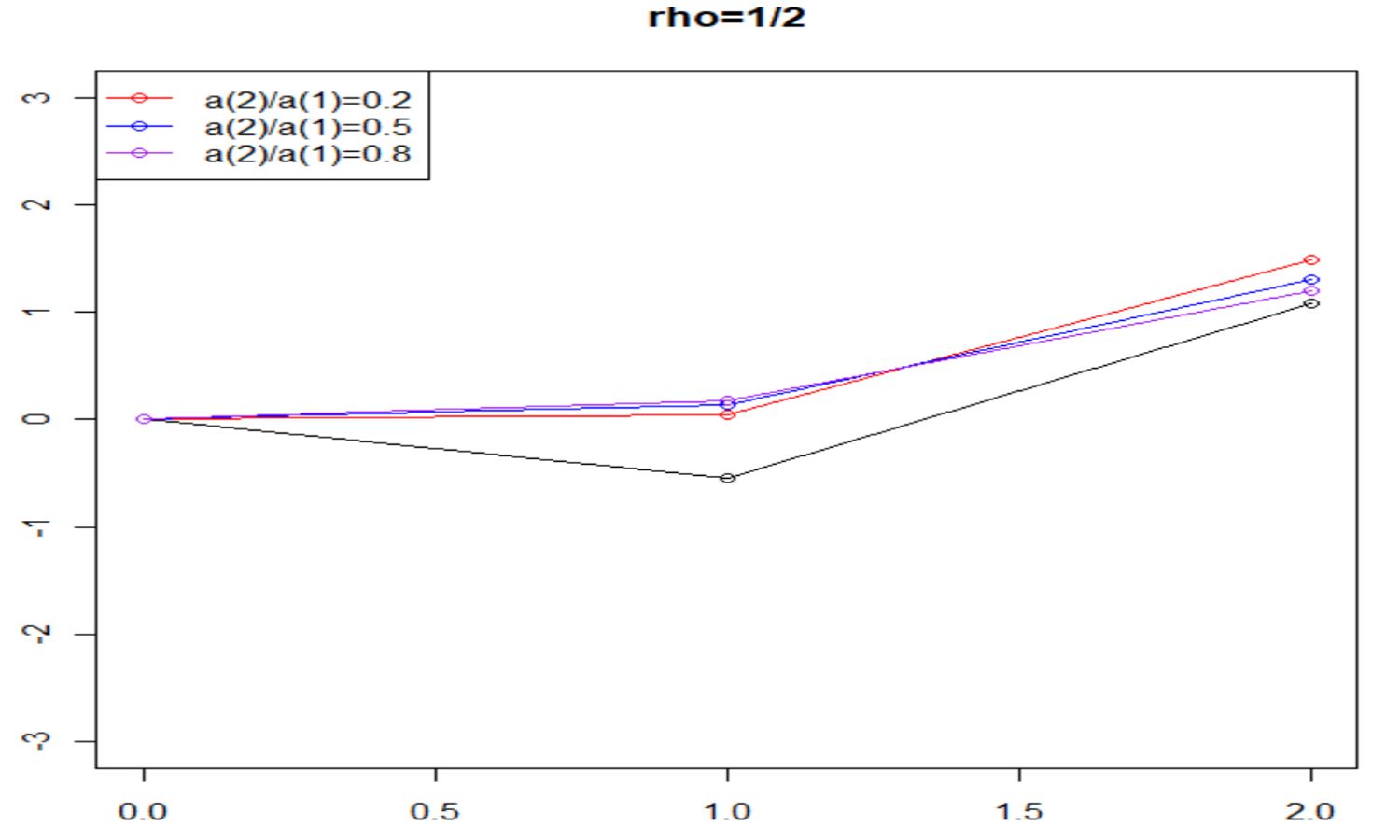}
    \caption{Filtering corresponding to different weight functions in the case $\rho=1/2$}
    \label{fig:enter-label}
\end{figure}

\begin{figure}
    \centering
    \includegraphics[width=1\linewidth]{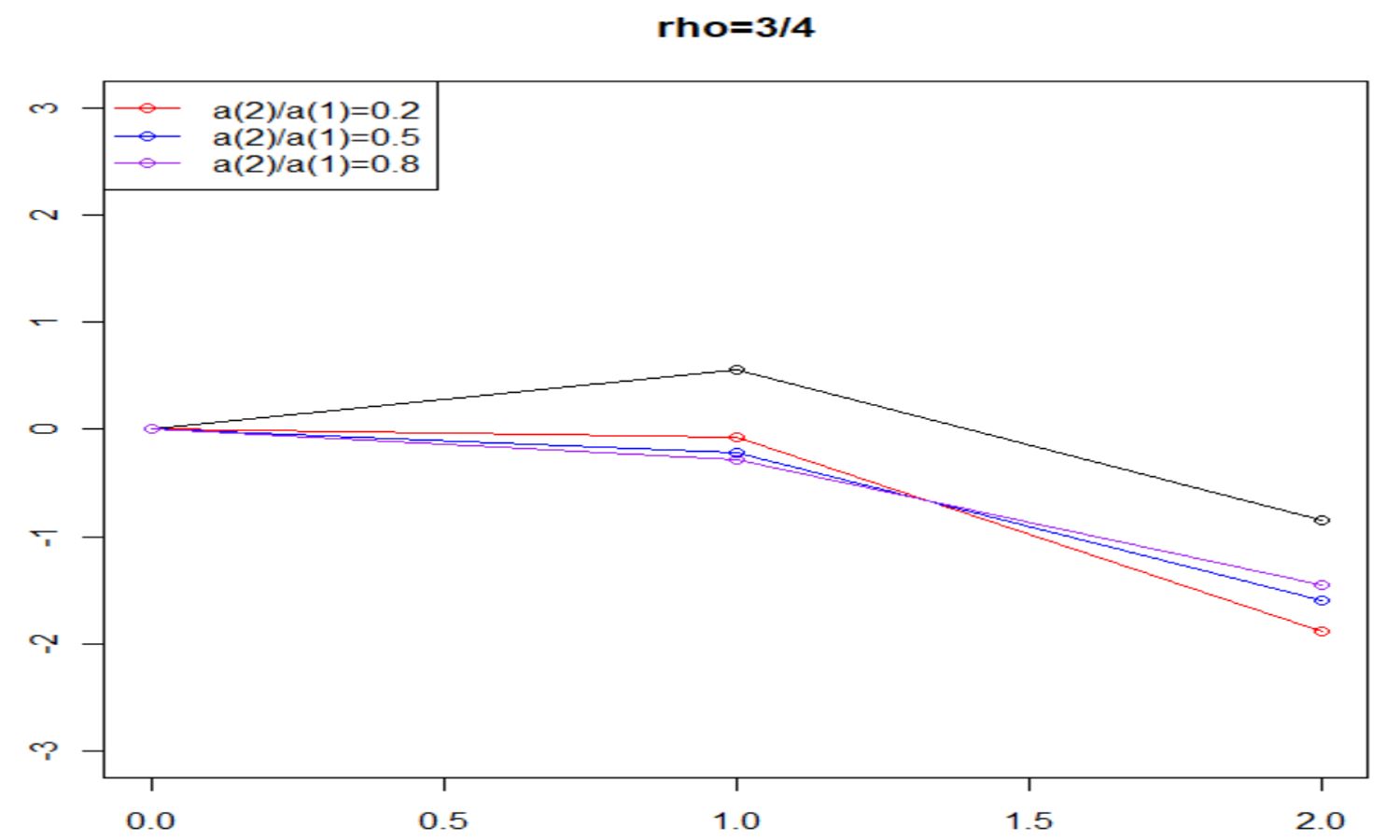}
    \caption{Filtering corresponding to different weight functions in the case $\rho=3/4$}
    \label{fig:enter-label}
\end{figure}
These simulations show that the optimal filtering is different when the weight function is different. Furthermore, the larger $a(i)$ is, the better the estimation of $x(i)$ is.

\FloatBarrier
\bibliography{main}


\begin{thebibliography}{20}
\ifx \bisbn   \undefined \def \bisbn  #1{ISBN #1}\fi
\ifx \binits  \undefined \def \binits#1{#1}\fi
\ifx \bauthor  \undefined \def \bauthor#1{#1}\fi
\ifx \batitle  \undefined \def \batitle#1{#1}\fi
\ifx \bjtitle  \undefined \def \bjtitle#1{#1}\fi
\ifx \bvolume  \undefined \def \bvolume#1{\textbf{#1}}\fi
\ifx \byear  \undefined \def \byear#1{#1}\fi
\ifx \bissue  \undefined \def \bissue#1{#1}\fi
\ifx \bfpage  \undefined \def \bfpage#1{#1}\fi
\ifx \blpage  \undefined \def \blpage #1{#1}\fi
\ifx \burl  \undefined \def \burl#1{\textsf{#1}}\fi
\ifx \doiurl  \undefined \def \doiurl#1{\url{https://doi.org/#1}}\fi
\ifx \betal  \undefined \def \betal{\textit{et al.}}\fi
\ifx \binstitute  \undefined \def \binstitute#1{#1}\fi
\ifx \binstitutionaled  \undefined \def \binstitutionaled#1{#1}\fi
\ifx \bctitle  \undefined \def \bctitle#1{#1}\fi
\ifx \beditor  \undefined \def \beditor#1{#1}\fi
\ifx \bpublisher  \undefined \def \bpublisher#1{#1}\fi
\ifx \bbtitle  \undefined \def \bbtitle#1{#1}\fi
\ifx \bedition  \undefined \def \bedition#1{#1}\fi
\ifx \bseriesno  \undefined \def \bseriesno#1{#1}\fi
\ifx \blocation  \undefined \def \blocation#1{#1}\fi
\ifx \bsertitle  \undefined \def \bsertitle#1{#1}\fi
\ifx \bsnm \undefined \def \bsnm#1{#1}\fi
\ifx \bsuffix \undefined \def \bsuffix#1{#1}\fi
\ifx \bparticle \undefined \def \bparticle#1{#1}\fi
\ifx \barticle \undefined \def \barticle#1{#1}\fi
\bibcommenthead
\ifx \bconfdate \undefined \def \bconfdate #1{#1}\fi
\ifx \botherref \undefined \def \botherref #1{#1}\fi
\ifx \url \undefined \def \url#1{\textsf{#1}}\fi
\ifx \bchapter \undefined \def \bchapter#1{#1}\fi
\ifx \bbook \undefined \def \bbook#1{#1}\fi
\ifx \bcomment \undefined \def \bcomment#1{#1}\fi
\ifx \oauthor \undefined \def \oauthor#1{#1}\fi
\ifx \citeauthoryear \undefined \def \citeauthoryear#1{#1}\fi
\ifx \endbibitem  \undefined \def \endbibitem {}\fi
\ifx \bconflocation  \undefined \def \bconflocation#1{#1}\fi
\ifx \arxivurl  \undefined \def \arxivurl#1{\textsf{#1}}\fi
\csname PreBibitemsHook\endcsname

\bibitem[\protect\citeauthoryear{Kalman and Bucy}{1961}]{kalman1961neww}
\begin{barticle}
\bauthor{\bsnm{Kalman}, \binits{R.E.}},
\bauthor{\bsnm{Bucy}, \binits{R.S.}}:
\batitle{New results in linear filtering and prediction theory}.
\bjtitle{Journal of Basic Engineering}
\bvolume{83}(\bissue{1}),
\bfpage{95}--\blpage{108}
(\byear{1961})
\end{barticle}
\endbibitem

\bibitem[\protect\citeauthoryear{Chui and Chen}{2017}]{chui2017kalman}
\begin{botherref}
\oauthor{\bsnm{Chui}, \binits{C.K.}},
\oauthor{\bsnm{Chen}, \binits{G.}}:
Kalman filtering.
Springer
(2017)
\end{botherref}
\endbibitem

\bibitem[\protect\citeauthoryear{Lawrence and Kaufman}{1971}]{lawrence1971kalman}
\begin{barticle}
\bauthor{\bsnm{Lawrence}, \binits{R.}},
\bauthor{\bsnm{Kaufman}, \binits{H.}}:
\batitle{The kalman filter for the equalization of a digital communications channel}.
\bjtitle{IEEE Transactions on Communication Technology}
\bvolume{19}(\bissue{6}),
\bfpage{1137}--\blpage{1141}
(\byear{1971})
\end{barticle}
\endbibitem

\bibitem[\protect\citeauthoryear{Masreliez}{1975}]{masreliez1975approximate}
\begin{barticle}
\bauthor{\bsnm{Masreliez}, \binits{C.}}:
\batitle{Approximate non-gaussian filtering with linear state and observation relations}.
\bjtitle{IEEE Transactions on Automatic Control}
\bvolume{20}(\bissue{1}),
\bfpage{107}--\blpage{110}
(\byear{1975})
\end{barticle}
\endbibitem

\bibitem[\protect\citeauthoryear{Tang et~al.}{2020}]{tang2020robot}
\begin{barticle}
\bauthor{\bsnm{Tang}, \binits{M.}},
\bauthor{\bsnm{Chen}, \binits{Z.}},
\bauthor{\bsnm{Yin}, \binits{F.}}:
\batitle{Robot tracking in slam with masreliez-martin unscented kalman filter}.
\bjtitle{International Journal of Control, Automation and Systems}
\bvolume{18}(\bissue{9}),
\bfpage{2315}--\blpage{2325}
(\byear{2020})
\end{barticle}
\endbibitem

\bibitem[\protect\citeauthoryear{Kushner}{1964}]{kushner1964dynamical}
\begin{barticle}
\bauthor{\bsnm{Kushner}, \binits{H.J.}}:
\batitle{On the dynamical equations of conditional probability density functions, with applications to optimal stochastic control theory}.
\bjtitle{Journal of Mathematical Analysis and Applications}
\bvolume{8}(\bissue{2}),
\bfpage{332}--\blpage{344}
(\byear{1964})
\end{barticle}
\endbibitem

\bibitem[\protect\citeauthoryear{Kushner}{1967}]{kushner1967dynamical}
\begin{barticle}
\bauthor{\bsnm{Kushner}, \binits{H.J.}}:
\batitle{Dynamical equations for optimal nonlinear filtering}.
\bjtitle{Journal of Differential Equations}
\bvolume{3}(\bissue{2}),
\bfpage{179}--\blpage{190}
(\byear{1967})
\end{barticle}
\endbibitem

\bibitem[\protect\citeauthoryear{Stratonovich}{1960}]{stratonovich1960conditional}
\begin{barticle}
\bauthor{\bsnm{Stratonovich}, \binits{R.}}:
\batitle{Conditional markov processes}.
\bjtitle{Theory of Probability \& Its Applications}
\bvolume{5}(\bissue{2}),
\bfpage{156}--\blpage{178}
(\byear{1960})
\end{barticle}
\endbibitem

\bibitem[\protect\citeauthoryear{Fujisaki et~al.}{1972}]{Fujisaki1972stochastic}
\begin{barticle}
\bauthor{\bsnm{Fujisaki}, \binits{M.}},
\bauthor{\bsnm{Kallianpur}, \binits{G.}},
\bauthor{\bsnm{Kunita}, \binits{H.}}:
\batitle{Stochastic differential equations for the non linear filtering problem}.
\bjtitle{Osaka Journal of Mathematics}
\bvolume{9}(\bissue{1}),
\bfpage{19}--\blpage{40}
(\byear{1972})
\end{barticle}
\endbibitem

\bibitem[\protect\citeauthoryear{Kallianpur and Striebel}{1968}]{kallianpur1968estimation}
\begin{barticle}
\bauthor{\bsnm{Kallianpur}, \binits{G.}},
\bauthor{\bsnm{Striebel}, \binits{C.}}:
\batitle{Estimation of stochastic systems: Arbitrary system process with additive white noise observation errors}.
\bjtitle{The Annals of Mathematical Statistics}
\bvolume{39}(\bissue{3}),
\bfpage{785}--\blpage{801}
(\byear{1968})
\end{barticle}
\endbibitem

\bibitem[\protect\citeauthoryear{Kallianpur and Striebel}{1969}]{kallianpur1969stochastic}
\begin{barticle}
\bauthor{\bsnm{Kallianpur}, \binits{G.}},
\bauthor{\bsnm{Striebel}, \binits{C.}}:
\batitle{Stochastic differential equations occurring in the estimation of continuous parameter stochastic processes}.
\bjtitle{Theory of Probability \& Its Applications}
\bvolume{14}(\bissue{4}),
\bfpage{567}--\blpage{594}
(\byear{1969})
\end{barticle}
\endbibitem

\bibitem[\protect\citeauthoryear{Duncan}{1967}]{duncan1967probability}
\begin{botherref}
\oauthor{\bsnm{Duncan}, \binits{T.E.}}:
Probability densities for diffusion processes with applications to nonlinear filtering theory and detection theory.
Stanford University
(1967)
\end{botherref}
\endbibitem

\bibitem[\protect\citeauthoryear{Mortensen}{1966}]{mortensen1966optimal}
\begin{botherref}
\oauthor{\bsnm{Mortensen}, \binits{R.E.}}:
Optimal control of continuous-time stochastic systems.
University of California, Berkeley
(1966)
\end{botherref}
\endbibitem

\bibitem[\protect\citeauthoryear{Zakai}{1969}]{zakai1969optimal}
\begin{barticle}
\bauthor{\bsnm{Zakai}, \binits{M.}}:
\batitle{On the optimal filtering of diffusion processes}.
\bjtitle{Zeitschrift f{\"u}r Wahrscheinlichkeitstheorie und verwandte Gebiete}
\bvolume{11}(\bissue{3}),
\bfpage{230}--\blpage{243}
(\byear{1969})
\end{barticle}
\endbibitem

\bibitem[\protect\citeauthoryear{Mandelbrot and Van~Ness}{1968}]{mandelbrot1968fractional}
\begin{barticle}
\bauthor{\bsnm{Mandelbrot}, \binits{B.B.}},
\bauthor{\bsnm{Van~Ness}, \binits{J.W.}}:
\batitle{Fractional brownian motions, fractional noises and applications}.
\bjtitle{SIAM review}
\bvolume{10}(\bissue{4}),
\bfpage{422}--\blpage{437}
(\byear{1968})
\end{barticle}
\endbibitem

\bibitem[\protect\citeauthoryear{Linn and Amirdjanova}{2009}]{linn2009representations}
\begin{barticle}
\bauthor{\bsnm{Linn}, \binits{M.}},
\bauthor{\bsnm{Amirdjanova}, \binits{A.}}:
\batitle{Representations of the optimal filter in the context of nonlinear filtering of random fields with fractional noise}.
\bjtitle{Stochastic processes and their applications}
\bvolume{119}(\bissue{8}),
\bfpage{2481}--\blpage{2500}
(\byear{2009})
\end{barticle}
\endbibitem

\bibitem[\protect\citeauthoryear{Su et~al.}{2011}]{su2011multiscale}
\begin{barticle}
\bauthor{\bsnm{Su}, \binits{L.}},
\bauthor{\bsnm{Zhang}, \binits{Y.}},
\bauthor{\bsnm{Ma}, \binits{Y.}},
\bauthor{\bsnm{Li}, \binits{J.}},
\bauthor{\bsnm{Li}, \binits{F.}}:
\batitle{Multiscale kf algorithm for strong fractional noise interference suppression in discrete-time uwb systems}.
\bjtitle{Discrete Dynamics in Nature and Society}
\bvolume{2011}(\bissue{1}),
\bfpage{356421}
(\byear{2011})
\end{barticle}
\endbibitem

\bibitem[\protect\citeauthoryear{Afterman et~al.}{2022}]{afterman2022linear}
\begin{barticle}
\bauthor{\bsnm{Afterman}, \binits{D.}},
\bauthor{\bsnm{Chigansky}, \binits{P.}},
\bauthor{\bsnm{Kleptsyna}, \binits{M.}},
\bauthor{\bsnm{Marushkevych}, \binits{D.}}:
\batitle{Linear filtering with fractional noises: large time and small noise asymptotics}.
\bjtitle{SIAM Journal on Control and Optimization}
\bvolume{60}(\bissue{3}),
\bfpage{1463}--\blpage{1487}
(\byear{2022})
\end{barticle}
\endbibitem

\bibitem[\protect\citeauthoryear{Ahmed and Li}{1991}]{ahmed1991quadratic}
\begin{barticle}
\bauthor{\bsnm{Ahmed}, \binits{N.}},
\bauthor{\bsnm{Li}, \binits{P.}}:
\batitle{Quadratic regulator theory and linear filtering under system constraints}.
\bjtitle{IMA Journal of Mathematical Control and Information}
\bvolume{8}(\bissue{1}),
\bfpage{93}--\blpage{107}
(\byear{1991})
\end{barticle}
\endbibitem

\bibitem[\protect\citeauthoryear{Ahmed and Charalambous}{2002}]{ahmed2002filtering}
\begin{barticle}
\bauthor{\bsnm{Ahmed}, \binits{N.U.}},
\bauthor{\bsnm{Charalambous}, \binits{C.D.}}:
\batitle{Filtering for linear systems driven by fractional brownian motion}.
\bjtitle{SIAM Journal on Control and Optimization}
\bvolume{41}(\bissue{1}),
\bfpage{313}--\blpage{330}
(\byear{2002})
\end{barticle}
\endbibitem

\end{thebibliography}
\end{document}